\documentclass[11pt,reqno]{amsart}
\usepackage{amssymb,latexsym,amsmath,epsfig,amsthm,mathrsfs}
\usepackage{rotating}
\usepackage{graphicx}
\usepackage{amssymb}
\usepackage{lineno}
\usepackage{enumitem}
\usepackage{cite}
\usepackage{multicol}
\usepackage[top=3cm, bottom=3cm, left=2.5cm, right=2.5cm]{geometry}
\usepackage[usenames]{color}
\usepackage{float}
\usepackage[colorlinks=true,
linkcolor=blue,
filecolor=blue,
citecolor=red]{hyperref}
\numberwithin{equation}{section}

\newtheorem{question}{Question}[section]

\newtheorem{corollary}{Corollary}[section]

\newtheorem{theorem}{Theorem}[section]
\newtheorem{lemma}[theorem]{Lemma}
\newtheorem{definition}[theorem]{Definition}
\newtheorem{remark}[theorem]{Remark}
\usepackage{thmtools, thm-restate}
\newtheorem{conjecture}[theorem]{Conjecture}

\begin{document}
\title[Exact and Asymptotic Counts of MSTD, MDTS, and Balanced Sets in Dicyclic Groups]{Exact and Asymptotic Counts of MSTD, MDTS, and Balanced Sets in Dicyclic Groups}

\author[Sagar Mandal]{Sagar Mandal}
\address{Department of Mathematics\\Indian Institute of Technology Ropar\\Punjab, India.}
\email{sagar.25maz0008@iitrpr.ac.in, sagarmandal31415@gmail.com}
 
\author[Neetu]{Neetu$^{*}$}
	\address{Department of Mathematical and Computational Sciences, National Institute of Technology Surathkal, Karnataka, 575025, India}
	\email{chananianeetu@gmail.com}
    \thanks{$^{*}$The corresponding author}


\begin{abstract}
We investigate the relationship between the sizes of the sum and difference sets of the Dicyclic Group $\mathrm{Dic}_{4n}$. We first determine the exact numbers of MSTD (more sums than differences), MDTS (more differences than sums), and balanced subsets of size two. As a consequence, we show that the numbers of MSTD and balanced subsets of size two are asymptotically equal as $n \to \infty$. For odd $n$, we then obtain exact counts of MSTD, MDTS, and balanced subsets of size three, with the results depending on whether $n$ is divisible by $3$. In this case, we establish that asymptotically the number of MSTD subsets of size three is six times the number of MDTS subsets and also six times the number of balanced subsets. Finally, we establish a lower bound for the number of MSTD, MDTS, and balanced subsets of $\mathrm{Dic}_{4n}$ corresponding to the boundary case of size $2n$. 
\end{abstract}
\maketitle
\tableofcontents
\section{Introduction}
 If $A\subset \mathbb{Z}$, where $\mathbb{Z}$ denotes the group of integers,
 the sumset and difference set of $A$ are defined as follows.
 $$A+A:=\{x\in \mathbb{Z}: x=a_1+a_2  \text{ for some }  a_1,a_2 \in A\},$$
 $$A-A:=\{x\in \mathbb{Z}: x=a_1-a_2 \text{ for some } a_1,a_2 \in A\}.$$
 These elementary operations are fundamental in additive number theory. For integers $a$ and $b$, we define $[a,b] = \{x\in \mathbb{Z}: a\leq x \leq b\}$. A natural problem of
recent interest has been to understand the relative sizes of the sum and difference sets of a set  $A$.
\begin{definition}
If $|A+A| > |A-A|$, $A$ is called a  {\emph{more-sums-than-differences (MSTD)}} set or a {\emph{sum-dominant}}  set , while if $|A+A| = |A-A| $,  $A$ is said to be   {\emph{balanced}}, and if
$|A+A| < |A-A|$, then $A$ is called a {\emph{more-differences-than-sums (MDTS)}}
set or a  {\emph{difference-dominant}} set.
\end{definition}
Since addition is commutative and subtraction is not, we usually expect most sets to be MDTS. However, MSTD sets do exist. It is believed that Conway in 1969  gave the first example of an MSTD set, which is $\{0, 2, 3, 4, 7, 11, 12, 14\}$.  Martin and  O'Bryant \cite{Martin_Bryant_2006} proved that for all $n\geq 15$, a positive proportion of the $2^{n}$ subsets of $\{0, 1, \ldots, n-1\}$ are MSTD. Zhao \cite{Zhao_2011} later gave a deterministic algorithm to compute the proportion of MSTD
sets as $n$ goes to infinity and proved that this proportion is at least $4.28\cdot10^{-4}$. 

The last few years have seen an explosion of papers examining the properties of sum-dominant sets of integers, one can see (\cite{Hegarty_2007, Hegarty_Miller_2009, Chu_2022, Kim_Miller_2022, Miller_Scheinerman_2010, Miller_Peterson_2019, Asada_2017, Lazarev_Miller_OBryant_2013, Iyer_Lazarev_Miller_Zhang_2012, Martin_Bryant_2006,  Do_Kulkarni_Miller_Moon_Wellens_Wilcox_2015, Zhao_2010a}).\par
Much of the study of sum-dominant sets has concerned subsets of the
integers; however, the phenomenon in finite  groups has received
some attention, notably from  \cite{ Hegarty_2007}, \cite{Nathanson_2007},   \cite{Zhao_2010} and \cite{Penman_2014}. One key motivation for this line of research is the close connection between MSTD sets in finite abelian groups and those in the integers. Nathanson in \cite{Nathanson_2007} showed that families of MSTD sets of integers can be constructed from MSTD sets in finite abelian groups. More precisely, if
 $A\subset \mathbb{Z}/n\mathbb{Z}$ satisfies $|A+A|>|A-A|$, then $\{a\in\mathbb{Z} :a \text{ mod } n\in A,0\leq a\leq kn\}$ is an MSTD set of integers for sufficiently large $k$. \par
 In this direction, Nathanson in \cite{Nathanson_2007} established lower bounds on the number of MSTD subsets in certain finite abelian groups, showing that
 \begin{equation*}  
 |\text{MSTD}(\mathbb{Z}/n\mathbb{Z}\times\mathbb{Z}/2\mathbb{Z})|\geq
  \begin{cases}
   2^n(1-\frac{2n}{2^{n/2}}),  & \text{if n is even}, \\
       2^n(1-\frac{\sqrt{2}n}{2^{n/2}}),  & \text{if n is odd}.
      \end{cases}
    \end{equation*} 
Zhao in \cite{Zhao_2010} improved and generalized Nathanson’s result. He gave asymptotics for
$|\mathrm{MSTD}(G)|$ for large finite abelian groups $G$. In particular, he showed that $$
|\mathrm{MSTD}(\mathbb{Z}/n\mathbb{Z})| \sim
\begin{cases}
3^{n/2}, & \text{if $n$ is even}, \\
\frac{1}{2}\, n^{\varphi ^{n}}, & \text{if $n$ is odd}.
\end{cases}
$$ and $$
|\mathrm{MSTD}(\mathbb{Z}/n\mathbb{Z} \times \mathbb{Z}/2\mathbb{Z})| \sim
\begin{cases}
3^{n+1}, & \text{if $n$ is even}, \\
3^{n}, & \text{if $n$ is odd}.
\end{cases}
$$
In the context of non-abelian groups, Miller and Vissuet in  \cite{Miller_Vissuet_2014} began a systematic investigation of sumsets and difference sets, started with dihedral group  
$$D_{2n} := \langle r, s : r^n = s^2 = 1,\; r^{-1}s = sr \rangle.$$

Based on both theoretical and computational evidence, they proposed the following conjecture.
\begin{conjecture}\textup{\cite{Miller_Vissuet_2014}}\label{millerconjecture} Let $n\geq3 $ be an integer. There are more MSTD subsets of $D_{2n}$ than MDTS subsets of $D_{2n}.$\end{conjecture} 
Further progress toward this conjecture was made by Ascoli et al. in \cite{Ascoli et al_2022}, who studied subsets of $D_{2n}$ according to their cardinality. 

 In the finite groups, although the operation is usually written multiplicatively, we follow the notation from  \cite{Miller_Vissuet_2014} and define the sumset and difference set as  $$A+A=\{a_1a_2: a_1,a_2\in A\},$$ and $$A-A =\{a_1a_2^{-1}: a_1,a_2 \in A\}.$$  
 In this paper we focus on dicyclic groups. For a positive integer $n\geq 1$, the dicyclic group, denoted as $\mathrm{Dic}_{4n}$, is defined as $$\mathrm{Dic}_{4n}= \langle a,b : a^{2n}=1, a^{n}=b^2, b^{-1}ab=a^{-1}\rangle.$$
 The following properties of the dicyclic group $\mathrm{Dic}_{4n}$ can be derived from the relations $a^{2n}=1$, $a^{n}=b^2$, and $ab=ba^{-1}$, as shown in \cite{Cheng_Feng_Huang_2019}.
\begin{lemma}\textup{\cite[Lemma 2.6]{Cheng_Feng_Huang_2019}}
    For the dicyclic group $\mathrm{Dic}_{4n}$, we have 
    \begin{enumerate}
        \item[\upshape(1)] $ba^{k}=a^{-k}b, a^{k}b=ba^{-k};$
        \item[\upshape(2)] $a^{k}ba^{m}b=a^{k-m+n};$
        \item[\upshape(3)] $(ba^{k})^{-1}=ba^{n+k}, (a^{k}b)^{-1}=a^{n+k}b$,
    \end{enumerate} where $k\in [0,2n-1]$.
\end{lemma}
Consequently, the group $\mathrm{Dic}_{4n}$ admits the following explicit representation:
 $$\mathrm{Dic}_{4n}=\{1, a,a^{2},\ldots, a^{2n-1}, b, ab, a^2b, \ldots, a^{2n-1}b \}.$$ 
 We now compare  MSTD and MDTS sets in the dihedral groups $D_{2n}$ and the dicyclic groups $\mathrm{Dic}_{4n}$ for small values of $n$.
When $n=1$, both groups are isomorphic to $\mathbb{Z}/4\mathbb{Z}$ and all subsets of $\mathbb{Z}/4\mathbb{Z}$ are balanced. For larger values of $n$, the structures begin to diverge, and interesting behaviors emerge. We present below the number of MSTD and MDTS sets of size $m$ in $D_{2n}$ and $\mathrm{Dic}_{4n}$, for $2\leq n\leq 5$ and $2\leq m\leq 10$.

\begin{table}[h]
\centering
\scriptsize
\caption{Number of MSTD sets in  $D_{2n}$ and  $\mathrm{Dic}_{4n}$.}\label{table 1}
\begin{tabular}{ |p{0.8cm}||*{10}{p{0.9cm}|} }
 \hline
 Size & $D_8$ & $\mathrm{Dic}_8$ & $D_{12}$ & $\mathrm{Dic}_{12}$ & $D_{16}$ & $\mathrm{Dic}_{16}$ & $D_{20}$& $\mathrm{Dic}_{20}$  \\
 \hline
2 & 8 & 0 & 24 & 24 & 48 & 32 & 80 & 80  \\
3 & 24 & 0 & 132 & 24 & 352 & 96 & 760 & 400  \\
4 & 32 & 16 & 246 & 126 & 1208 & 392 & 3380 & 1380  \\
5 & 0 & 0 & 456 & 120 & 2896 & 752 & 11440 & 3840  \\
6 & 0 & 0 & 276 & 60 & 5304 & 1368 & 27200 & 9520  \\
7 & 0 & 0 & 0 & 0 & 5504 & 992 & 51400 & 14480 \\
8 & 0 & 0 & 0 & 0 & 1888 & 288 & 69320 & 11920  \\
9 & 0 & 0 & 0 & 0 & 0 & 0 & 43200 & 5360 \\
10 & 0 & 0 & 0 & 0 & 0 & 0 & 9960 & 1080 \\

 \hline
\end{tabular}
\end{table}
\begin{table}[h]
\centering
\scriptsize
\caption{Number of MDTS sets in  $D_{2n}$ and  $\mathrm{Dic}_{4n}$.}\label{table 2}
\begin{tabular}{ |p{0.8cm}||*{10}{p{0.9cm}|} }
 \hline
 Size & $D_8$ & $\mathrm{Dic}_8$ & $D_{12}$ & $\mathrm{Dic}_{12}$ & $D_{16}$ & $\mathrm{Dic}_{16}$ & $D_{20}$ & $\mathrm{Dic}_{20}$  \\
 \hline
2 & 0 & 0 & 0 & 0 & 0 & 0 & 0 & 0  \\
3 & 0 & 24 & 12 & 48 & 32 & 160 & 80 & 200 \\
4 & 8 & 24 & 0 & 228 & 64 & 832 & 120 & 2080  \\
5 & 0 & 0 & 72 & 264 & 128 & 2016 & 680 & 7160 \\
6 & 0 & 0 & 48 & 48 & 784 & 3344 & 2280 & 19920  \\
7 & 0 & 0 & 0 & 0 & 640 & 1888 & 5600 & 31760  \\
8 & 0 & 0 & 0 & 0 & 104 & 296 & 10160 & 23520  \\
9 & 0 & 0 & 0 & 0 & 0 & 0 & 5560 & 7080  \\
10 & 0 & 0 & 0 & 0 & 0 & 0 & 960 & 480  \\

  \hline
\end{tabular}
\end{table} 
From these tables, it is clear that dihedral groups always contain more MSTD sets than MDTS sets, which supports Conjecture \ref{millerconjecture}. In fact, Ascoli et al. \cite{Ascoli et al_2022} proved that, for subsets of sizes $2$ and $3$, the dihedral group $D_{2n}$ contains strictly more MSTD sets than MDTS sets. However, this does not hold for dicyclic groups, which display different behavior.\par
For $m=2$, we can easily observe from  Table \ref{table 1} and Table \ref{table 2} that, dicyclic groups also contain strictly more MSTD sets than MDTS sets. In contrast, for $m=3$, we observe that when $2 \leq n \leq 4$, dicyclic groups have more MDTS sets than MSTD sets, whereas for $n=5$ the situation reverses and MSTD sets become more numerous. 
 
The aim of this paper is to investigate sumsets, difference sets, and balanced sets in dicyclic groups. Before stating our main results, we introduce some notation. 
For $k \in \mathbb{N}$, we denote by
$$
\mathrm{S}_{\mathrm{Dic}_{4n}}(k), \quad \mathrm{D}_{\mathrm{Dic}_{4n}}(k), \quad \mathrm{B}_{\mathrm{Dic}_{4n}}(k)
$$
the number of MSTD, MDTS, and balanced subsets of $\mathrm{Dic}_{4n}$ of cardinality $k$, respectively.

For a fixed subset $A \subseteq \mathrm{Dic}_{4n}$ \textit{representing a given type}, we write
$$
\mathrm{S}_{\mathrm{Dic}_{4n}}(k,A), \quad \mathrm{D}_{\mathrm{Dic}_{4n}}(k,A), \quad \mathrm{B}_{\mathrm{Dic}_{4n}}(k,A)
$$
for the corresponding counts restricted to \textit{subsets of type $A$}.

Moreover, for $0 \leq r \leq n$, we define the subsets
\begin{align*}
\mathcal{R}_{2r} = \{a^{2r}, a^{2r+1}, \dots, a^{2r+2n-1}\}, \quad \mathcal{F}_{2r} = \{a^{2r}b, a^{2r+1}b, \dots, a^{2r+2n-1}b\},
\end{align*}
where exponents are taken modulo $2n$. Note that $\mathrm{Dic}_{4n} = \mathcal{R}_{2r} \cup \mathcal{F}_{2r}$.

Our main results in this direction are stated as follows.
\begin{theorem}\label{Theorem1.0}
Let $n\geq2$ be an integer. Then the following hold.
\begin{enumerate}
    \item[\upshape(1)] If $n$ is even, then
    $$
    \mathrm{S}_{\mathrm{Dic}_{4n}}(2)=4n(n-2), \qquad
    \mathrm{D}_{\mathrm{Dic}_{4n}}(2)=0, \qquad
    \mathrm{B}_{\mathrm{Dic}_{4n}}(2)=2n(2n+3).
    $$
    \item[\upshape(2)] If $n$ is odd, then
    $$
    \mathrm{S}_{\mathrm{Dic}_{4n}}(2)=4n(n-1), \qquad
    \mathrm{D}_{\mathrm{Dic}_{4n}}(2)=0, \qquad
    \mathrm{B}_{\mathrm{Dic}_{4n}}(2)=2n(2n+1).
    $$
\end{enumerate}
\end{theorem}
\begin{remark}
From Theorem \ref{Theorem1.0}, we obtain, as $n \to \infty$,
$$
  \begin{aligned}
    \mathrm{S}_{\mathrm{Dic}_{4n}}(2) \sim \mathrm{B}_{\mathrm{Dic}_{4n}}(2).
 \end{aligned}$$
 In particular, for sufficiently large $n$, the number of MSTD sets is asymptotically equal to the number of balanced sets.
\end{remark}
\begin{theorem}\label{Theorem1.1}
Let $n\geq3$ be an odd integer. Then the following hold. 
\begin{enumerate}
    \item[\upshape(1)] If $3 \mid n$, then
    $$
    \begin{aligned}
    \mathrm{S}_{\mathrm{Dic}_{4n}}(3) &= 4n(n-1)(2n-5),\\
    \mathrm{D}_{\mathrm{Dic}_{4n}}(3) &= \frac{2n(2n^2+3n-3)}{3},\\
    \mathrm{B}_{\mathrm{Dic}_{4n}}(3) &= \frac{2n(2n^2+27n-25)}{3}.
    \end{aligned}
    $$
    \item[\upshape(2)] If $3 \nmid n$, then
    $$
    \begin{aligned}
    \mathrm{S}_{\mathrm{Dic}_{4n}}(3) &= 4n(n-1)(2n-5),\\
    \mathrm{D}_{\mathrm{Dic}_{4n}}(3) &= \frac{2n(n-1)(2n+5)}{3},\\
    \mathrm{B}_{\mathrm{Dic}_{4n}}(3) &= \frac{2n(2n^2+27n-23)}{3}.
    \end{aligned}
    $$
\end{enumerate}
\end{theorem}
\begin{remark}
Let $n\geq 3$ be an odd integer. From Theorem \ref{Theorem1.1}, we obtain, as $n \to \infty$,
$$
  \begin{aligned}
    \mathrm{S}_{\mathrm{Dic}_{4n}}(3) \sim 6 \mathrm{D}_{\mathrm{Dic}_{4n}}(3),\quad  \mathrm{S}_{\mathrm{Dic}_{4n}}(3) \sim 6 \mathrm{B}_{\mathrm{Dic}_{4n}}(3),\quad  \mathrm{D}_{\mathrm{Dic}_{4n}}(3) \sim  \mathrm{B}_{\mathrm{Dic}_{4n}}(3).
 \end{aligned}$$
 In particular, for sufficiently large odd $n$, the number of MSTD sets is asymptotically six times the number of MDTS or balanced sets.
\end{remark}

\begin{corollary}\label{Corollary1.2}
Let $n\geq 3$ be an odd integer and $\mathcal{A}_{4n,3}$ denote the collection of subsets of $\mathrm{Dic}_{4n}$ of size $3$. Then, except for $n=3$, the collection $\mathcal{A}_{4n,3}$ contains more MSTD sets than MDTS sets.
\end{corollary}

It has been proved in (\cite{gr12}, Theorem 2.1) that  for any finite group $G$, if $|A|>|G|/2$, then $|A+A|=|A-A|$, where $A$ is a finite subset of  $G$. This implies that if $|A|>2n$, then $A$ is a balanced set in $\mathrm{Dic}_{4n}$. Consequently, it is natural to investigate what happens in the boundary case $|A|=2n$. The following theorem provides some necessary bounds for $ \mathrm{S}_{\mathrm{Dic}_{4n}}(2n,A), \mathrm{D}_{\mathrm{Dic}_{4n}}(2n,A), \mathrm{B}_{\mathrm{Dic}_{4n}}(2n,A)$, for \textit{subsets of type $A$} whose exponents of the rotation elements form an arithmetic progression of length $n$ with common difference $1$.

\begin{theorem}\label{Theorem1.7}
Let $A=R \cup Sb$ be a subset of $\mathrm{Dic_{4n}}$, where  $$R=\{a^{r},a^{r+1},\ldots,a^{r+n-1}\}, \quad 0\leq r\leq n,$$  and $$S=\{a^{i_1},a^{i_2},\ldots, a^{i_n}\}\subset \mathcal{R}_0, \quad 0\leq i_1,i_2,\ldots, i_n\leq 2n-1.$$  Then the following hold.
\begin{enumerate}
    \item[\upshape(1)] If $n\geq6$ is an even integer, then
      $$\begin{aligned}
  \mathrm{S}_{\mathrm{Dic}_{4n}}(2n,A) &\geq (n+1)\bigl(7\cdot 2^{n-3}-4\bigr),\\[8pt] 
  \quad
   \mathrm{D}_{\mathrm{Dic}_{4n}}(2n,A)&\geq 2(n+1) ,\\[8pt]
   \quad
    \mathrm{B}_{\mathrm{Dic}_{4n}}(2n,A)&\geq (n+1)\sum_{\substack{i=3\\i\text{~odd}}}^{n+3}\binom{2n-i}{n-3} .
    \end{aligned}
    $$
    \item[\upshape(2)] If $n\geq6$ is an odd integer, then
      $$\begin{aligned}
    \mathrm{S}_{\mathrm{Dic}_{4n}}(2n,A) &\geq (n+1)\bigl(7\cdot 2^{n-3}-4\bigr),\\[8pt]
   \quad \mathrm{D}_{\mathrm{Dic}_{4n}}(2n,A)&=0,\\[8pt]
  \quad  \mathrm{B}_{\mathrm{Dic}_{4n}}(2n,A)&\geq\binom{2n}{n}(n+1)-2^{n}(n+1).
    \end{aligned}
    $$
\end{enumerate}
\end{theorem}
\begin{theorem}\label{Theorem 1.9}
  Let $n\geq6$ be an integer. Then the following hold.
\begin{enumerate}
    \item[\upshape(1)] If $n$ is even, then
    $$
    \mathrm{S}_{\mathrm{Dic}_{4n}}(2n)\geq (n+1)\bigl(7\cdot 2^{n-3}-4\bigr), \qquad
    \mathrm{D}_{\mathrm{Dic}_{4n}}(2n)\geq 2(n+1) , \qquad
    \mathrm{B}_{\mathrm{Dic}_{4n}}(2n)\geq (n+1)\sum_{\substack{i=3\\i\text{~odd}}}^{n+3}\binom{2n-i}{n-3} .
    $$
    \item[\upshape(2)] If $n$ is odd, then
    $$
    \mathrm{S}_{\mathrm{Dic}_{4n}}(2n)\geq (n+1)\bigl(7\cdot 2^{n-3}-4\bigr), \qquad
    \mathrm{D}_{\mathrm{Dic}_{4n}}(2n)\geq 0 , \qquad
    \mathrm{B}_{\mathrm{Dic}_{4n}}(2n)\geq  \binom{2n}{n}(n+1)-2^{n}(n+1).
    $$
\end{enumerate}  
\end{theorem}
To establish our main results, we first prove several auxiliary lemmas (see Section \ref{section-2}). The proofs of Theorems \ref{Theorem1.0}, \ref{Theorem1.1}, and Corollary \ref{Corollary1.2} are given in Sections \ref{section-3} and \ref{section-4}, respectively. In Section \ref{section-5}, we prove Theorems \ref{Theorem1.7} and  \ref{Theorem 1.9}. In Section \ref{section-6}, we discuss future directions and pose several open problems.

\section{Preliminaries}\label{section-2}
In this section, we prove several auxiliary results that will be used in the subsequent proofs. 
\begin{lemma}\label{lem:triples-mod2n}
Let $n\geq3$ be an odd integer, and let
$$
\mathcal{S}_n = \{(i,j,k) \in \mathbb{Z}^3 : 0 \le i < j < k \le 2n-1\}.
$$
Define $\mathcal{T}_n \subset \mathcal{S}_n$ to consist of those triples for which at least one of the following congruences holds modulo $2n$:
$$
2i \equiv j+k,\qquad 2j \equiv i+k,\qquad 2k \equiv i+j.
$$
Then
$$
|\mathcal{T}_n| =
\begin{cases}
2n(n-1), & 3 \nmid n,\\[8pt]
\dfrac{2n(3n-5)}{3}, & 3 \mid n.
\end{cases}
$$
\end{lemma}

\begin{proof} Let $(i,j,k) \in \mathcal{T}_n$. We first count the number of solutions to each congruence separately, and then subtract the overlaps.

Assume that \begin{equation}\label{condition-1}
    2i \equiv j+k \pmod{2n}.
\end{equation} Since $3 \le j + k \le 4n - 3$, it follows that either
$$j + k = 2i \quad \text{or} \quad j + k = 2i + 2n.$$
The equality $j + k = 2i$ is not possible because $k > j > i$. Hence, $$j+k=2i+2n,$$
 which  implies $k=2i+2n-j$. Since $j<k$, we obtain  
$j<2i+2n-j$, and therefore $j\leq i+n-1$. Moreover, as $k\leq 2n-1$, we get  $2i+2n-j\leq 2n-1$, and thus $j\geq  2i+1$. It follows that  
\begin{equation}\label{choices for j}
    2i+1 \leq j \leq i+n-1.
\end{equation}
Therefore, necessarily
$0\leq i\leq n-2$. Using \eqref{choices for j}, the number of choices for  $j$ is $n-i-1$, for each $i<j$. Hence, the total number of triples $(i,j,k)$ satisfying  \eqref{condition-1} is \begin{equation}\label{2.3.1}
    \sum_{i=0}^{n-2}(n-i-1)=\frac{n(n-1)}{2}.
\end{equation}

Assume that \begin{equation}\label{conditon-2}
    2j \equiv i+k \pmod{2n}.
\end{equation}  Since $k\leq 2n-1$, the case $i+k=2j+2n$ is not possible. Thus,  $i+k=2j$, which gives $$k=2j-i.$$
Using $k\leq 2n-1$ we get $2j-(2n-1)\leq i$. For $1\leq j\leq n-1$ we have $2j-(2n-1)\leq 0$,  therefore \begin{equation}\label{condition-2.1}
    0\leq i \leq j-1.
\end{equation}
For $n\leq j\leq 2n-2$ we have $i\geq 2j-(2n-1)\geq 1$, therefore \begin{equation}\label{condition-2.2}
    2j-(2n-1)\leq i \leq j-1.
\end{equation}

For  $1\leq j\leq n-1$, using \eqref{condition-2.1}, we get $j$  possible values for $i$, and for $n\leq j\leq 2n-2$, using \eqref{condition-2.2}, we get $2n-1-j$ possible values of $i$.  Therefore, the total number of $(i,j,k)$ is 
\begin{equation}\label{2.3.2}
    \sum_{j=1}^{n-1} j + \sum_{j=n}^{2n-2}(2n-1-j)=\frac{n(n-1)}{2}+\frac{n(n-1)}{2}= n(n-1).
\end{equation}
Assume that \begin{equation}\label{condition-3}
    2k \equiv i+j \pmod{2n}.
\end{equation} Since $i<j<k$, the case  $i+j=2k$ is not possible. Thus, $$i+j=2k-2n.$$ Moreover, since $i<j$, we have $2i<i+j$, and therefore $2i<2k-2n$, which implies $i\leq k-n-1$.
Thus, for fixed $k$, admissible values of $i$ exist if and only if  $k\geq n+1$. Consequently, the total number of triples  $(i,j,k)$ satisfying \eqref{condition-3} is
\begin{equation}\label{2.3.3}
   \sum_{k=n+1}^{2n-1}(k-n)=\frac{n(n-1)}{2}. 
\end{equation}
 Hence, the total number of triples  $(i,j,k)$ 
satisfying all three congruences \eqref{condition-1}, \eqref{conditon-2}, and \eqref{condition-3}, counted with multiplicity, is obtained by combining \eqref{2.3.1}, \eqref{2.3.2}, and \eqref{2.3.3}:  $$\frac{n(n-1)}{2}+n(n-1)+\frac{n(n-1)}{2}=2n(n-1).$$
This counts all solutions of the given congruences, allowing repetitions. However, we are interested only in distinct solutions. Observe that any two of the three congruences imply the third one. Therefore, it suffices to determine all triples $(i,j,k)$  that satisfy any two of the above congruences. This occurs if and only if
$$3(i-j)\equiv0 \pmod{2n}.$$
If $3\nmid n$, then $\gcd(3,2n)=1$, and hence  $i\equiv j\pmod{2n}$, which is
impossible. Therefore, in this case, no overlaps occur.\par
On the other hand, if $3\mid n$, then $\gcd(3,2n)=3$. Writing $n=3m$, the overlapping solutions are precisely of the form
$$(i,j,k)=(i, i+2m, i+4m),\qquad 0 \leq i \leq 2m-1.$$
There are exactly  $2m=\frac{2n}{3}$ such triples, and each of them satisfies all three congruences. Therefore, when  $3\mid n$, the number of distinct triples $(i,j,k)$ is
$$2n(n-1) - 2\cdot\frac{2n}{3}= 2n(n-1)-\frac{4n}{3}= \frac{2n(3n-5)}{3}.$$
This completes the proof.
\end{proof}
\begin{remark}\label{remark 2.3}
Let $n\geq6$ be an integer. If $\mathcal{S}\subset \mathbb{Z}/2n\mathbb{Z}$ is a subset with $|\mathcal{S}|=n$ satisfying $n\not\in \mathcal{S}-\mathcal{S}$. Then  $\mathcal{S}$ contains at most one element from each of the pairs 
$$(0,n),~(1,n+1),~(2,n+2),\ldots,(n-1,2n-1).$$
Thus, the total number of such subsets $\mathcal{S}$ is $2^n$.   
\end{remark}
\begin{lemma}\label{Lemma 2.2}
Let $n\geq6$ be an even integer and let $\mathcal{S}\subset \mathbb{Z}/2n\mathbb{Z}$ with $|\mathcal{S}|=n$ such that $n\not\in \mathcal{S}-\mathcal{S}$, where $\mathcal{S-S}=\{s_1-s_2\pmod{2n}: s_1,s_2\in \mathcal{S}\}$. Then for a fixed integer $0\leq r\leq n$, $2r+n-1\pmod{2n}\in \mathcal{S}-\mathcal{S}.$ 
\end{lemma}
\begin{proof} Let $\mathcal{S}\subset \mathbb{Z}/2n\mathbb{Z}$ be a subset satisfying $n\not\in \mathcal{S}-\mathcal{S}$. Then by Remark \ref{remark 2.3}, $\mathcal{S}$ contains at most one element from each of the pairs 
$$(0,n),~(1,n+1),~(2,n+2),\ldots,(n-1,2n-1).$$

Fix an integer $r$ with $0\leq r\leq n$, and define $l_{r}=2r+n-1\pmod{2n}.$
Then $l_r$ is an odd integer. To prove $l_{r}\in\mathcal{S}-\mathcal{S} $, it suffices to prove that $(\mathcal{S}+l_r)\cap\mathcal{S}\neq \emptyset$, where $\mathcal{S}+l_r=\{s+l_r \pmod{2n}: s\in \mathcal{S}\}$. Since $n$ is an even positive integer, it is easy to observe that  $\mathcal{S}$ contains half even and half odd integers. Write 
\begin{align*}
&\mathcal{S}=\{o_1,o_2,\ldots, o_{n/2},e_1,e_2,\ldots,e_{n/2}\},\\
   &\mathcal{E}_1=\{o_1+l_r\pmod{2n},\ldots, o_{n/2}+l_r\pmod{2n}\},\\ &\mathcal{E}_2=\{o_1-l_r\pmod{2n},\ldots,o_{n/2}-l_r\pmod{2n}\},\\ &\mathcal{E}_3=\{o_1-l_r+n\pmod{2n},\ldots,o_{n/2}-l_r+n\pmod{2n}\}.  
\end{align*}
 where $o_i,e_i$ are distinct odd and even integers, respectively. Since $n\not\in \mathcal{S}-\mathcal{S}$, we have $\mathcal{E}_2\cap \mathcal{E}_3=\emptyset$ and $|\mathcal{E}_2\cup \mathcal{E}_3|=n$, thus $\mathcal{E}_2\cup \mathcal{E}_3$ contains all even integers of $\mathbb{Z}/2n\mathbb{Z}$. Let
\begin{align}\label{2.10}
\mathcal{E}=\mathcal{E}_2\cup \mathcal{E}_3.      
 \end{align}
 If any $e_j\in \mathcal{E}_2$ then we are done as
$(\mathcal{S}+l_r)\cap\mathcal{S}\neq \emptyset$. Therefore, set $$\mathcal{S}=\{o_1,o_2,\ldots, o_{n/2},o_1-l_r+n\pmod{2n},\ldots,o_{n/2}-l_r+n\pmod{2n}\},$$    then $$\mathcal{S}+l_r=\{o_1+l_r\pmod{2n},\ldots, o_{n/2}+l_r\pmod{2n},o_1+n\pmod{2n},\ldots,o_{n/2}+n\pmod{2n}\}.$$    

If possible, assume that $(\mathcal{S}+l_r)\cap\mathcal{S} =\emptyset$, this forces $\mathcal{E}_1\cap \mathcal{E}_3=\emptyset$. Since $|\mathcal{E}_1\cup\mathcal{E}_3|=n$ we have $\mathcal{E}=\mathcal{E}_1\cup \mathcal{E}_3$. From (\ref{2.10}) we have $\mathcal{E}_1=\mathcal{E}_2$.
Therefore, we must have
$$
    \sum_{i\in \mathcal{E}_1} i\equiv \sum_{i\in \mathcal{E}_2}i\pmod{2n},
$$
it follows that
$$nl_r\equiv 0\pmod{2n}$$
but this forces $2\mid l_r$, which is a contradiction. Hence $( \mathcal{S}+l_r)\cap\mathcal{S}\neq \emptyset$. This completes the proof.
\end{proof}

\begin{lemma}\label{Lemma 2.3}
Let $n\geq6$ be an odd integer and $\mathcal{S}\subset \mathbb{Z}/2n\mathbb{Z}$ with $|\mathcal{S}|=n$. Then for a fixed integer $0\leq r\leq n$, $2r+n-1\pmod{2n}\in \mathcal{S}-\mathcal{S},$  where $\mathcal{S-S}=\{s_1-s_2\pmod{2n}: s_1,s_2\in \mathcal{S}\}$.
\end{lemma}
\begin{proof}
Let $2r+n-1\pmod{2n}=l_r$, then $l_r$ is an even integer. Since $n$ is an odd positive integer, $\mathcal{S}$ contains at least $\frac{n+1}{2}$ even integers or at least $\frac{n+1}{2}$ odd integers. Without loss of generality, assume that $\mathcal{S}$ contains at least $\frac{n+1}{2}$ even integers. Then $\mathcal{S}+l_r=\{s+l_r\pmod{2n}:s\in \mathcal{S}\}$ also contains at least $\frac{n+1}{2}$ even integers. It follows that $(\mathcal{S}+l_r)\cap\mathcal{S}\neq \emptyset$. This completes the proof.
\end{proof}
\begin{lemma}\label{Lemma 2.4}
Let $n\geq6$ be an integer and also let $\mathcal{S}\subset \mathbb{Z}/2n\mathbb{Z}$ with $|\mathcal{S}|=n$ such that $n\in \mathcal{S}-\mathcal{S}$. Then for any fixed $r\in \mathbb{Z}$ we have
$$\{r+t+s \pmod{2n}:0\leq t\leq n-1,\ s\in \mathcal{S}\}=\mathbb{Z}/2n\mathbb{Z}.$$
\end{lemma}
\begin{proof}
Since $n\in \mathcal{S}-\mathcal{S}$, we have some $s_1,s_2\in \mathcal{S}$ such that $s_2-s_1\equiv n \pmod{2n}$. Without loss of generality, suppose that $s_2=s_1+n$. Then 
$\{r+t+s_1 \pmod{2n},r+t+s_1+n\pmod{2n}:0\leq t\leq n-1\}=\mathbb{Z}/2n\mathbb{Z}$, therefore, the proof follows.
\end{proof}
\begin{lemma}\label{Lemma 2.5}
Let $n\geq6$ be an integer and also let $\mathcal{S}\subset \mathbb{Z}/2n\mathbb{Z}$ be a subset with $|\mathcal{S}|=n$ such that $n\not\in \mathcal{S}-\mathcal{S}$. Then, for any fixed $0\leq r\leq n$, there are at least $7\cdot 2^{n-3}-4$ such sets $\mathcal{S}$ for which 
$$\{r+t+s\pmod{2n}:0\leq t\leq n-1,\ s\in \mathcal{S}\}=\mathbb{Z}/2n\mathbb{Z}.$$
\end{lemma}
\begin{proof}
Let $\mathcal{S}=\{i_1,i_2,\ldots,i_n\}\subset  \mathbb{Z}/2n\mathbb{Z}$ be a subset such that $n\not\in \mathcal{S}-\mathcal{S}$, and also let $R=\{r+t+s\pmod{2n}:0\leq t\leq n-1,\ s\in \mathcal{S}\}$. Then $\mathcal{S}$ contains at most one element from each of the pairs 
$$(0,n),~(1,n+1),~(2,n+2),\ldots,(n-1,2n-1).$$
Without loss of generality, assume $i_1\in \{0,n\},i_2\in \{1,n+1\},i_3\in\{2,n+2\}$. Thus there are exactly eight possible choices for the triple $(i_1,i_2,i_3)$. We analyze these cases.
\begin{enumerate}
    \item If $\{0,n+1,2\}\subset\mathcal{S}$, then $R=\mathbb{Z}/2n\mathbb{Z}$. There are exactly $2^{n-3}$ such sets.
    \item If $\{0,n+1,n+2\}\subset\mathcal{S}$, then $\{r+t+s \pmod{2n}:0\leq t\leq n-1,\ s=0,n+1,n+2\}=(r+\mathbb{Z}/2n\mathbb{Z})\setminus\{r+n \pmod{2n}\}$. Suppose  there exists an element $i_j\in\mathcal{S}$ which  is the first element of one of the pairs
$$(3,n+3),(4,n+4),\ldots,(n-1,2n-1),$$
then $r+n\pmod{2n}\in\{r+t+s \pmod{2n}: 0\leq t\leq n-1,\ s=i_j \,\}$, and hence $R=\mathbb{Z}/2n\mathbb{Z}$. Thus, except for the single exceptional set $\mathcal{S}=\{0,n+1,n+2,n+3,n+4,\ldots, 2n-1\}$, we obtain $R=\mathbb{Z}/2n\mathbb{Z}$. Therefore, we have $2^{n-3}-1$ such desired sets.

\item If $\{0,1,n+2\}\subset\mathcal{S}$, then $\{r+t+s\pmod{2n}:0\leq t\leq n-1,\ s=0,1,n+2\}=(r+\mathbb{Z}/2n\mathbb{Z})\setminus\{r+n+1\pmod{2n}\}$. Moreover, the missing element $r+n+1\pmod{2n}\in R$ whenever  $\mathcal{S}\cap \{3,4,\ldots, n-1\}\neq\emptyset$. Thus, except for the single exceptional set $\mathcal{S}=\{0,1,n+2,n+3,n+4,\ldots, 2n-1\}$, we obtain $R=\mathbb{Z}/2n\mathbb{Z}$. Therefore, we have $2^{n-3}-1$ such desired sets.

\item If $\{n,1,n+2\}\subset\mathcal{S}$, then $R=\mathbb{Z}/2n\mathbb{Z}$. There are exactly $2^{n-3}$ such sets.

\item If $\{n,1,2\}\subset\mathcal{S}$, then $\{r+t+s\pmod{2n}:0\leq t\leq n-1,\ s=n,1,2\}=(r+\mathbb{Z}/2n\mathbb{Z})\setminus\{r\}$. Now, if there exists  an element $i_j\in\mathcal{S}$ that appears as  the second element of one of the pairs
$$(3,n+3),(4,n+4),\ldots,(n-1,2n-1),$$
then $r\in\{r+t+s\pmod{2n} : 0\leq t\leq n-1,\ s=i_j \,\},$ and hence $R=\mathbb{Z}/2n\mathbb{Z}$. Thus, except for the single exceptional set $\mathcal{S}=\{n,1,2,3,4,\ldots, n-1\}$ we obtain $R=\mathbb{Z}/2n\mathbb{Z}$. Therefore, we have $2^{n-3}-1$ such desired sets.

\item If $\{n,n+1,2\}\subset\mathcal{S}$,  then $\{r+t+s\pmod{2n}:0\leq t\leq n-1,\ s=n,n+1,2\}=(r+\mathbb{Z}/2n\mathbb{Z})\setminus\{r+1\}$.  Further, the missing element $r+1\in R$ whenever $S\cap \{n+3,n+4,\ldots,2n-1\}\neq \emptyset$. Therefore, except for the single exceptional set $\mathcal{S}=\{n,n+1,2,3,4,\ldots, n-1\}$, we obtain $R=\mathbb{Z}/2n\mathbb{Z}$. Therefore, we have $2^{n-3}-1$ such desired sets.

\item  If $\{0,1,2\}\subset\mathcal{S}$, then $\{r+t+s\pmod{2n}:0\leq t\leq n-1,\ s=0,1,2\}=(r+\mathbb{Z}/2n\mathbb{Z})\setminus\{r+n+2\pmod{2n},\ldots, r+2n-1\pmod{2n}\}$. To get $R=\mathbb{Z}/2n\mathbb{Z}$, we choose sets $\mathcal{S}$ such that
$$\{0,1,2,n+3,4\}\subset \mathcal{S},\quad \{0,1,2,3,n+4,5\}\subset \mathcal{S} , \quad \{0,1,2,n+3,n+4,5\}\subset \mathcal{S}.$$
The total number of such sets is $2^{n-4}$.

\item If $\{n,n+1,n+2\}\subset\mathcal{S}$, then $\{r+t+s\pmod{2n}:0\leq t\leq n-1,\ s=n,n+1,n+2\}=(r+\mathbb{Z}/2n\mathbb{Z})\setminus\{r+2,\ldots, r+n-1\}$. To get $R=\mathbb{Z}/2n\mathbb{Z}$, we choose sets $\mathcal{S}$ such that
$$\{n,n+1,n+2,3,n+4\}\subset\mathcal{S},\quad \{n,n+1,n+2,n+3,4,n+5\}\subset \mathcal{S}$$ and $$\{n,n+1,n+2,3,4,n+5\}\subset \mathcal{S}.$$
The total number of such sets is $2^{n-4}$. 
\end{enumerate}
Therefore the number of subsets $\mathcal{S}$ for which $R=\mathbb{Z}/2n\mathbb{Z}$ is at least $7\cdot 2^{n-3}-4$.
\end{proof}
\section{Proof of Theorem \ref{Theorem1.0}}\label{section-3}
\begin{proof}[Proof of Theorem \ref{Theorem1.0}]
The subsets of $\mathrm{Dic_{4n}}$ of size $2$ are precisely the sets
$$\begin{aligned}
A&=\{a^i,a^j\}~~\text{where}~~0\leq i<j\leq 2n-1,\\
B&=\{a^i,a^jb\}~~\text{where}~~0\leq i,j\leq 2n-1,\\
C&=\{a^ib,a^jb\}~~\text{where}~~0\leq i<j\leq 2n-1.\end{aligned}$$

\textbf{Case 1}($A=\{a^i,a^j\}$). Here we have $$A+A=\{a^{2i},a^{i+j},a^{2j}\},$$
$$A-A=\{1,a^{2n+i-j},a^{2n+j-i}\}.$$
It is straightforward to verify that collisions in $A+A\text{ and }A-A$ occur precisely when $i\equiv j \pmod{n}$, that is, when $(i,j)\in \{(i,n+i): 0\leq i\leq n-1\}$, for such $(i,j)$ we have $|A+A|=|A-A|=2$, otherwise $|A+A|=|A-A|=3$. Therefore, we get
$$\mathrm{S_{\mathrm{Dic_{4n}}}(2,A)}=0,\quad \mathrm{D_{\mathrm{Dic_{4n}}}(2,A)}=0,\quad \mathrm{B_{\mathrm{Dic_{4n}}}(2,A)}=\sum_{i=0}^{2n-2}(2n-1-i)=n(2n-1).$$
\textbf{Case 2}($B=\{a^i,a^jb\}$). In this case, we get 

$$B+B=\{a^{2i},a^{i+j}b,a^{j-i}b,a^n\}$$
$$B-B=\{1,a^{i+j+n}b,a^{i+j}b\}.$$
Collisions in $B+B$ occur precisely when   $$2i\equiv n \pmod{2n} \text{ or }i+j\equiv j-i \pmod{2n},$$ that is, when $2i\equiv n \pmod{2n} \text{ or }2i\equiv 0\pmod{2n}$. If no collisions occur we get MSTD sets. \par
\textbf{Subcase 2.1}($n$ odd).
If $n$ is odd, then $2i\equiv n \pmod{2n}$ has no solution, while $2i\equiv 0 \pmod{2n}$ is satisfied for $i=0,n$. Thus, for  $(i,j)\in \{(0,j),(n,j): 0\leq j\leq 2n-1\}$, we have $$|B+B|=|B-B|=3.$$ Therefore, we have
$$\mathrm{S_{\mathrm{Dic_{4n}}}(2,B)}=4n^2-4n=4n(n-1),\quad \mathrm{D_{\mathrm{Dic_{4n}}}(2,B)}=0,\quad \mathrm{B_{\mathrm{Dic_{4n}}}(2,B)}=4n.$$\par
\textbf{Subcase 2.2}($n$ even).
If $n$ is even then $2i\equiv n \pmod{2n}$ holds for  $i=n/2,3n/2$ and $2i\equiv 0 \pmod{2n}$ holds for $i=0,n$. Thus, for  $(i,j)\in \{(0,j),(n/2,j),(n,j),(3n/2,j): 0\leq j\leq 2n-1\}$, we obtain $|B+B|=|B-B|=3$. Therefore, we get 
$$\mathrm{S_{\mathrm{Dic_{4n}}}(2,B)}=4n^2-8n=4n(n-2),\quad \mathrm{D_{\mathrm{Dic_{4n}}}(2,B)}=0,\quad \mathrm{B_{\mathrm{Dic_{4n}}}(2,B)}=8n.$$

\noindent\textbf{Case 3}$(C=\{a^ib,a^jb\})$. In this case, we have 
$$C+C=\{a^{n},a^{n+i-j},a^{n+j-i}\}$$ 
$$C-C=\{1,a^{i-j},a^{j-i}\}.$$ Any collision in  $|C+C|$ necessarily produces a corresponding collision in $|C-C|$, and hence 
$|C-C|=|C+C|$. Therefore, we get 
$$\mathrm{S_{\mathrm{Dic_{4n}}}(2,C)}=0,\quad \mathrm{D_{\mathrm{Dic_{4n}}}(2,C)}=0,\quad \mathrm{B_{\mathrm{Dic_{4n}}}(2,C)}=\sum_{i=0}^{2n-2}(2n-1-i)=n(2n-1).$$
The proof follows from Cases 1, 2, and 3.
\end{proof}
\section{Proof of Theorem \ref{Theorem1.1} and Corollary \ref{Corollary1.2}} \label{section-4}
The proof of Theorem \ref{Theorem1.1} requires the following auxiliary lemmas.

\begin{lemma}\label{Lemma1.4}
Let $n\geq3$ be an odd integer and let 
$$A=\{a^{i},a^{j},a^{k}\}\subseteq \{1,a,a^{2},\ldots,a^{2n-1},\, b,ab,a^{2}b,\ldots,a^{2n-1}b\},$$ where $0 \leq i<j< k\leq 2n-1$. Then the number of sets of type $A$ which are MSTD, MDTS, and balanced is given as follows.
\begin{enumerate}
    \item[\upshape(1)] If $3\nmid n$, then
$$\mathrm{S_{\mathrm{Dic_{4n}}}(3,A)}=0,\qquad\mathrm{D_{\mathrm{Dic_{4n}}}(3,A)}=\frac{4n(n-1)(n-2)}{3},\qquad\mathrm{B_{\mathrm{Dic_{4n}}}(3,A)}=2n(n-1).$$
\item[\upshape(2)] If $3\mid n$, then
$$ \mathrm{S_{\mathrm{Dic_{4n}}}(3,A)}=0,\qquad \mathrm{D_{\mathrm{Dic_{4n}}}(3,A)}=\frac{4n(n^{2}-3n+3)}{3},\qquad\mathrm{B_{\mathrm{Dic_{4n}}}(3,A)}=\frac{2n(3n-5)}{3}.$$
\end{enumerate}
\end{lemma}
\begin{proof}
 Let $A=\{a^{i},a^{j},a^{k}\}$ with $0 \leq i<j< k\leq 2n-1$. Then
$$A+A=\{a^{2i},a^{i+j},a^{i+k},a^{2j},a^{2k},a^{j+k}\},$$ 
$$A-A=\{1,a^{i-j},a^{i-k},a^{j-i},a^{j-k},a^{k-i},a^{k-j}\}.$$ Thus, in the absence of any collisions, we have $|A+A|<|A-A|$. Hence, any reduction in these sizes can occur only when there are collisions among the elements of $A+A$ or $A-A$.
A direct computation shows that such collisions in $A+A\text{ and }A-A$ occur precisely in the following cases:
\begin{align*}
&\text{(i)}\; 2i \equiv j+k \pmod{2n}, \quad
&\text{(ii)}\; 2j \equiv i+k \pmod{2n}, \quad
&\text{(iii)}\; 2k \equiv i+j \pmod{2n},\\
&\text{(iv)}\; i \equiv j \pmod{n}, \qquad
&\text{(v)}\; i \equiv k \pmod{n}, \qquad
&\text{(vi)}\; j \equiv k \pmod{n}. 
\end{align*}
If $2i \equiv j+k\pmod{2n}$, then $$A+A=\{a^{2i},a^{i+j},a^{i+k},a^{2j},a^{2k}\} \text{ and } A-A=\{1,a^{i-j},a^{i-k},a^{j-k},a^{k-j}\},$$ so that $|A+A|=|A-A|$. Additional collisions, such as  $i+j\equiv 2k$ or $i+k\equiv 2j$ or both, may further occur, but in each such scenario the set $A$  remains balanced. The same conclusion holds for the cases $2j\equiv i+k \pmod{2n}$ and $2k \equiv i+j \pmod{2n}$. Therefore, $A$ is balanced in the first three cases.  The number of balanced sets arising from these three cases can be obtained directly from Lemma \ref{lem:triples-mod2n}. Consequently
$$ \mathrm{B_{\mathrm{Dic_{4n}}}(3,A)}=\begin{cases}
2n(n-1), & \text{if } 3\nmid n,\\[6pt]
\frac{2n(3n-5)}{3}, & \text{if } 3\mid n.
\end{cases}.$$
Now consider the case $ i \equiv j \pmod{n}$. Then $$A+A=\{a^{2i},a^{2i+n},a^{i+k},a^{2k},a^{n+i+k}\}, \qquad A-A=\{1,a^{n},a^{i-k},a^{n+i-k},a^{k-i},a^{k-n-i}\}.$$  Thus $|A+A|=5$ and $|A-A|=6$, so  $A$ is an MDTS set. In particular, no MSTD or balanced set occurs in this case. Similarly, for the cases $ i \equiv k \pmod{n}$ and $ j \equiv k \pmod{n}$, no MSTD or balanced set occurs. Therefore, we obtain 
$$\qquad \mathrm{S_{\mathrm{Dic_{4n}}}(3,A)}=0$$ 
and  $$\mathrm{D_{\mathrm{Dic_{4n}}}(3,A)}=\binom{2n}{3}-\mathrm{S_{\mathrm{Dic_{4n}}}(3,A)}-\mathrm{B_{\mathrm{Dic_{4n}}}(3,A)}=\begin{cases}
\frac{2n(n-1)(2n-1)}{3}-2n(n-1), & \text{if } 3\nmid n,\\[6pt]
\frac{2n(n-1)(2n-1)}{3}-\frac{2n(3n-5)}{3}, & \text{if } 3\mid n.
\end{cases}$$

This completes the proof.
\end{proof}
\begin{lemma}\label{Lemma1.5}
Let $n\geq3$ be an odd integer and let 
$$B=\{a^i,a^j,a^kb\}\subset \{1,a,a^{2},\ldots,a^{2n-1},\, b,ab,a^{2}b,\ldots,a^{2n-1}b\},$$ where $0\leq i<j\leq 2n-1,~0\leq k\leq 2n-1$. Then the number of sets of type $B$ that are MSTD, MDTS, and balanced is given as follows.
$$\mathrm{S_{\mathrm{Dic_{4n}}}(3,B)}=2n(n-1)(2n-5),\qquad\mathrm{D_{\mathrm{Dic_{4n}}}(3,B)}=2n(n-1),\qquad\mathrm{B_{\mathrm{Dic_{4n}}}(3,B)}=2n(5n-4).$$
\end{lemma}
\begin{proof}
Let $B=\{a^i,a^j,a^kb\}$ with $0\leq i<j\leq 2n-1 \text{ and }~0\leq k\leq 2n-1$. Then 
$$B+B=\{a^{2i},a^{i+j},a^{2j},a^n\}\cup\{a^{i+k}b,a^{j+k}b,a^{k-i}b,a^{k-j}b\}$$  
$$B-B=\{1,a^{i-j},a^{j-i}\}\cup \{a^{n+i+k}b,a^{n+j+k}b,a^{k+i}b,a^{k+j}b\}.$$ 
Thus, in the absence of any collisions, we have $|B+B|>|B-B|$.
It is straightforward to verify that collisions in $B+B \text{ and }B-B$ occur precisely in the following cases:
\begin{align*}
&\text{(1)}\; i \equiv j \pmod{n}, \quad
&\text{(2)}\; i+j \equiv n \pmod{2n}, \quad
&\text{(3)}\; 2i \equiv 0 \pmod{2n},\\
&\text{(4)}\; i+j \equiv 0 \pmod{2n}, \quad
&\text{(5)}\; 2j \equiv 0 \pmod{2n}, \quad
&\text{(6)}\; n+i \equiv j \pmod{2n},\\
&\text{(7)}\; n+j \equiv i \pmod{2n}.
\end{align*}
Cases 6 and 7 are equivalent to Case 1. Note that the collisions do not depend on the value of $0\leq k\leq 2n-1$. Let $B_{r}$ denote the subset type corresponding to Case $r$, for $r\in \{1,2,3,4,5\}$.   We examine each case individually.\par
\textbf{Case 1}($i \equiv j \pmod n$).
In this case, we necessarily have $(i,j)\in \{(i,i+n): 0\leq i\leq n-1\}$. Consider subsets of the form
$$B_1=\{a^{i},a^{i+n},a^{k}b\}.$$

We first determine when $B_1$ is balanced. This occurs only in the case $(i,j)=(0,n)$. Indeed, when $i=0$, we obtain
$$B_1+B_1=\{1,a^{n}\}\cup\{a^{k}b,a^{n+k}b\},$$
$$B_1-B_1=\{1,a^{n}\}\cup\{a^{n+k}b,a^{k}b\}.$$
Thus $|B_1+B_1|=|B_1-B_1|$, and hence $B_1$ is a balanced set. Since $k$ ranges over $0\le k\le 2n-1$, this yields exactly $2n$ balanced sets.

Now assume $(i,j)\in \{(i,i+n): 1\leq i\leq n-1\}$. In this case,
$$B_1+B_1=\{a^{2i},a^{2i+n},a^{n}\}\cup\{a^{i+k}b,a^{i+n+k}b,a^{k-i}b,a^{k-i-n}b\},$$
whereas
$$B_1-B_1=\{1,a^{n}\}\cup\{a^{n+i+k}b,a^{i+k}b\}.
$$
Consequently, $|B_1+B_1|>|B_1-B_1|$, and thus $B_1$ is an MSTD set. Therefore, no MDTS sets arise in this case.

Finally, for each $i$ with $1\le i\le n-1$ and each $k$ with $0\le k\le 2n-1$, the corresponding set $B_1$ is MSTD. Hence the total number of MSTD sets in this case is
$$\sum_{i=1}^{n-1}2n=2n(n-1).$$
Therefore, in this case, we have 
\begin{equation}\label{eq:case1-counts}
\mathrm{S}_{\mathrm{Dic}_{4n}}(3,B_1)=2n(n-1), \qquad
\mathrm{D}_{\mathrm{Dic}_{4n}}(3,B_1)=0, \qquad
\mathrm{B}_{\mathrm{Dic}_{4n}}(3,B_1)=2n.
\end{equation}
\textbf{Case 2 ($i+j\equiv n \pmod{2n}$).} This implies 
$$(i,j)\in\{(i,n-i): 0\le i\le \frac{n-1}{2}\},$$
and $$(i,j)\in \{(i,3n-i) :n+1\le i\le \frac{3n-1}{2}\}.$$
Since the case $(i,j)=(0,n)$ has already been treated in Case~1, we exclude it from further consideration. For $(i,j)\in\{(i,n-i): 1\le i\le \frac{n-1}{2}\}\cup \{(i,3n-i) :n+1\le i\le \frac{3n-1}{2}\}$, we get 
$$B_2+B_2=\{a^{2i},a^{n},a^{2n-2i}\}\cup
\{a^{i+k}b,a^{n-i+k}b,a^{k-i}b,a^{k-n+i}b\},
$$
and 
$$B_2-B_2=\{1,a^{2i-n},a^{n-2i}\}\cup
\{a^{n+i+k}b,a^{k-i}b,a^{k+i}b,a^{k+n-i}b\}.
$$ In this case we have $|B_2+B_2|=|B_2-B_2|$, and hence every such set $B_2$ is balanced.  Therefore, the total number of balanced sets for each $0\leq k \leq 2n-1$ is
$$\sum_{i=1}^{\frac{n-1}{2}}2n+\sum_{i=n+1}^{\frac{3n-1}{2}}2n
=2n(n-1).
$$
Thus, in this case, we get 
\begin{equation}\label{eq:case2-counts}
\mathrm{S}_{\mathrm{Dic}_{4n}}(3,B_2)=0, \qquad
\mathrm{D}_{\mathrm{Dic}_{4n}}(3,B_2)=0, \qquad
\mathrm{B}_{\mathrm{Dic}_{4n}}(3,B_2)=2n(n-1).
\end{equation}
\textbf{Case 3}$(2i\equiv 0 \pmod{2n})$. 
In this case, $i\in \{0,n\}$. Hence  $$(i,j)\in \{(0,j):1\leq j\leq 2n-1\},$$ or  $$(i,j)\in\{(n,j): n+1\leq j\leq 2n-1\}.$$ Since the case $(i,j)=(0,n)$ has already been analyzed in Case~1,we exclude it from further consideration.  
If $(i,j)\in \{(0,j):1\leq j\leq 2n-1,j\neq n\},$ then  
$$B_{3}+B_{3}=\{1,a^{j},a^{2j},a^n\}\cup\{a^{k}b,a^{j+k}b,a^{k-j}b\},$$ 
$$B_{3}-B_{3}=\{1,a^{2n-j},a^{j}\}\cup \{a^{n+k}b,a^{n+j+k}b,a^{k}b,a^{k+j}b\}.$$
If $(i,j)\in\{(n,j): n+1\leq j\leq 2n-1\}$, then
$$B_{3}+B_{3}=\{1,a^{n+j},a^{2j},a^n\}\cup\{a^{n+k}b,a^{j+k}b,a^{k-j}b\}$$ 
$$B_{3}-B_{3}=\{1,a^{n-j},a^{j-n}\}\cup \{a^{k}b,a^{n+j+k}b,a^{k+n}b,a^{k+j}b\}.$$
In both cases, $|B_3+B_3|=|B_3-B_3|$, and hence every such set $B_3$ is balanced. Therefore, no MSTD or MDTS sets arise in this case. 
For each admissible value of $j$ and for each $k$ with $0\le k\le 2n-1$, we obtain a balanced set. Thus, the total number of balanced sets is
$$
\sum_{\substack{j=1 \\ j\neq n}}^{2n-1}2n
+\sum_{j=n+1}^{2n-1}2n
=6n(n-1).
$$
Therefore, in this case, we obtain
\begin{equation}\label{eq:case3-counts}
\mathrm{S}_{\mathrm{Dic}_{4n}}(3,B_3)=0, \qquad 
\mathrm{D}_{\mathrm{Dic}_{4n}}(3,B_3)=0, \qquad 
\mathrm{B}_{\mathrm{Dic}_{4n}}(3,B_3)=6n(n-1).
\end{equation}\par
\textbf{Case 4}$(i+j\equiv 0 \pmod{2n})$. In this case, we have  $(i,j)\in \{(i,2n-i): 1\leq i\leq n-1\}$. For such pairs $(i,j)$,  we compute 
$$B_{4}+B_{4}=\{1,a^{2i},a^{2n-2i},a^n\}\cup\{a^{i+k}b,a^{k-i}b\},$$ 
$$B_{4}-B_{4}=\{1,a^{2i},a^{2n-2i}\}\cup \{a^{n+i+k}b,a^{n-i+k}b,a^{k+i}b,a^{k-i}b\}.$$
Here, we have $|B_{4}-B_{4}|>|B_{4}+B_{4}|$, and hence every such set $B_4$ is an MDTS set. Consequently, no MSTD or balanced sets arise in this case. For each $i$ with $1\le i\le n-1$ and each $k$ with $0\le k\le 2n-1$, the total number of MDTS sets is
$$
\sum_{i=1}^{n-1}2n=2n(n-1).
$$ Therefore, we obtain \begin{equation}\label{eq:case4-counts}
\mathrm{S}_{\mathrm{Dic}_{4n}}(3,B_4)=0, \qquad 
\mathrm{D}_{\mathrm{Dic}_{4n}}(3,B_4)=2n(n-1), \qquad 
\mathrm{B}_{\mathrm{Dic}_{4n}}(3,B_4)=0.
\end{equation}\par
\textbf{Case 5}($2j\equiv 0 \pmod{2n}$). In this case, we have  $(i,j)\in\{(i,n):0\leq i \leq n-1\}$. Since the case $(i,j)=(0,n)$ has already been analyzed,  we exclude it from further consideration. For pairs $(i,j)\in\{(i,n):1\leq i \leq n-1\}$, we compute
$$B_5+B_5=\{a^{2i},a^{i+n},1,a^{n}\}\cup
\{a^{i+k}b,a^{n+k}b,a^{k-i}b\},$$
$$B_5-B_5=\{1,a^{i-n},a^{n-i}\}\cup
\{a^{n+i+k}b,a^{k}b,a^{k+i}b,a^{k+n}b\}.$$
Here, we have $|B_5+B_5|=|B_5-B_5|$, and hence every such set $B_5$ is balanced. Consequently, no MSTD or MDTS sets arise in this case. For each $i$ with $1\le i\le n-1$ and for each $k$ with $0\le k\le 2n-1$,  total number of balanced sets is
$$\sum_{i=1}^{n-1}2n=2n(n-1).$$
Therefore, in this case we obtain
\begin{equation}\label{eq:case5-counts}
\mathrm{S}_{\mathrm{Dic}_{4n}}(3,B_5)=0, \qquad 
\mathrm{D}_{\mathrm{Dic}_{4n}}(3,B_5)=0, \qquad 
\mathrm{B}_{\mathrm{Dic}_{4n}}(3,B_5)=2n(n-1).
\end{equation}
Therefore, combining \eqref{eq:case1-counts}, \eqref{eq:case2-counts}, \eqref{eq:case3-counts}, \eqref{eq:case4-counts} and \eqref{eq:case5-counts} we get 
\begin{align*}
\mathrm{D_{\mathrm{Dic_{4n}}}(3,B)}&=\mathrm{D_{\mathrm{Dic_{4n}}}(3,B_1)}+\mathrm{D_{\mathrm{Dic_{4n}}}(3,B_2)}+\mathrm{D_{\mathrm{Dic_{4n}}}(3,B_3)}+\mathrm{D_{\mathrm{Dic_{4n}}}(3,B_4)}+\mathrm{D_{\mathrm{Dic_{4n}}}(3,B_5)}\\&=0+0+0+2n(n-1)+0=2n(n-1),\end{align*}
\begin{align*}
\mathrm{B_{\mathrm{Dic_{4n}}}(3,B)}&=\mathrm{B_{\mathrm{Dic_{4n}}}(3,B_1)}+\mathrm{B_{\mathrm{Dic_{4n}}}(3,B_2)}+\mathrm{B_{\mathrm{Dic_{4n}}}(3,B_3)}+\mathrm{B_{\mathrm{Dic_{4n}}}(3,B_4)}+\mathrm{B_{\mathrm{Dic_{4n}}}(3,B_5)}\\&=2n+2n(n-1)+6n(n-1)+0+2n(n-1)=2n(5n-4).\end{align*}
Since the total number of subsets of type $B$ is $\binom{2n}{2}\times 2n$. Hence, 
\begin{align*}
    \mathrm{S}_{\mathrm{Dic}_{4n}}(3,B)&= \binom{2n}{2}\times 2n- \mathrm{D_{\mathrm{Dic_{4n}}}(3,B)}-\mathrm{B_{\mathrm{Dic_{4n}}}(3,B)}\\&=2n^{2}(2n-1)-2n(n-1)-2n(5n-4)=2n(n-1)(2n-5).
\end{align*}
\end{proof}
\begin{lemma}\label{Lemma1.6}
Let $n\geq3$ be an odd integer and let 
$$C=\{a^i,a^jb,a^kb\}\subset \{1,a,a^{2},\ldots,a^{2n-1},\, b,ab,a^{2}b,\ldots,a^{2n-1}b\},$$ where $0\leq i\leq 2n-1,~0\leq j<k\leq 2n-1$. Then the number of sets of type $C$ that are MSTD, MDTS, and balanced is given as follows.
$$\mathrm{S_{\mathrm{Dic_{4n}}}(3,C)}=2n(n-1)(2n-5),\qquad\mathrm{D_{\mathrm{Dic_{4n}}}(3,C)}=4n(n-1),\qquad\mathrm{B_{\mathrm{Dic_{4n}}}(3,C)}=2n(4n-3).$$
\end{lemma}
\begin{proof}
The proof is similar to Lemma \ref{Lemma1.5}.
\end{proof}
\begin{lemma}\label{Lemma1.7}
Let $n\geq3$ be an odd integer and let 
$$D=\{a^{i}b,a^{j}b,a^{k}b\}\subset \{1,a,a^{2},\ldots,a^{2n-1},\, b,ab,a^{2}b,\ldots,a^{2n-1}b\},$$ where $0 \leq i<j< k\leq 2n-1$. Then the number of sets of type $D$ that are MSTD, MDTS, and balanced is given as follows.
$$\mathrm{S_{\mathrm{Dic_{4n}}}(3,D)}=0,\qquad\mathrm{D_{\mathrm{Dic_{4n}}}(3,D)}=0,\qquad\mathrm{B_{\mathrm{Dic_{4n}}}(3,D)}=\frac{2n(n-1)(2n-1)}{3}.$$
\end{lemma}
\begin{proof}
 Let $D=\{a^{i}b,a^{j}b,a^{k}b\}$, where $0 \leq i<j< k\leq 2n-1$. Then
$$D+D=\{a^{n},a^{n+i-j},a^{n+i-k},a^{n+j-i},a^{n+j-k},a^{n+k-i},a^{n+k-j}\},$$ 
$$D-D=\{1,a^{i-j},a^{i-k},a^{j-i},a^{j-k},a^{k-i},a^{k-j}\}.$$
Any overlap among the elements of $|D+D|$ leads to an overlap in  $|D-D|$, and hence
$|D-D|=|D+D|.$ Therefore, every such set $D$ is balanced.
\end{proof}
\begin{proof}[Proof of Theorem \ref{Theorem1.1}]
The subsets of $\mathrm{Dic}_{4n}$ of size $3$ are precisely the sets of the following four types:
$$
\begin{aligned}
A&=\{a^i,a^j,a^k\}, && 0\le i<j<k\le 2n-1,\\
B&=\{a^i,a^j,a^kb\}, && 0\le i<j\le 2n-1,\; 0\le k\le 2n-1,\\
C&=\{a^{i},a^jb,a^kb\}, && 0\le i\le 2n-1,\; 0\le j<k\le 2n-1,\\
D&=\{a^ib,a^jb,a^kb\}, && 0\le i<j<k\le 2n-1.
\end{aligned}
$$
The result now follows directly from Lemmas \ref{Lemma1.4}, \ref{Lemma1.5}, \ref{Lemma1.6}, and \ref{Lemma1.7}, which treat each of these cases.
\end{proof}

\begin{proof}[Proof of Corollary \ref{Corollary1.2}]
The proof is an immediate consequence of Theorem \ref{Theorem1.1}.
\end{proof}
\section{Proof of Theorems \ref{Theorem1.7} and \ref{Theorem 1.9}}\label{section-5}
\begin{proof}[Proof of Theorem \ref{Theorem1.7}]
Let $A=R \cup Sb$ be a subset of $\mathrm{Dic_{4n}}$. Then 
$$A+A=(R+R)\cup (R+Sb)\cup (Sb+R)\cup (Sb+Sb),$$
$$A-A=(R-R)\cup (R-Sb)\cup (Sb-R)\cup (Sb-Sb).$$
Since  $R=\{a^{r},a^{r+1},\ldots,a^{r+n-1}\} \text{ with } 0\leq r\leq n,$  and $S=\{a^{i_1},\ldots, a^{i_n}\}\subset \mathcal{R}_0, \text{ with } 0\leq i_1,i_2,\ldots, i_n\leq 2n-1$, we obtain   
\begin{align*}
R+R&=\mathcal{R}_{2r}\setminus\{a^{2r+2n-1}\}, \text{ and } R-R=\mathcal{R}_0\setminus\{a^{n}\},
\\
R+Sb&=\{a^{r+t+i_j}b:0\leq t\leq n-1,\ 1\leq j\leq n\},\\
Sb+R&=\{a^{i_j-r-t}b:0\leq t\leq n-1,\ 1\leq j\leq n\},
\\
R-Sb&=\{a^{r+t+n+i_j}b:0\leq t\leq n-1,\ 1\leq j\leq n\},
\\
Sb-R&=\{a^{i_j+r+t}b:0\leq t\leq n-1,\ 1\leq j\leq n\},
\\
Sb+Sb&=\{a^{i_j-i_k+n}:1\leq j,k\leq n\},
\\
Sb-Sb&=\{a^{i_j-i_k}:1\leq j,k\leq n\}.
\end{align*}
It is easy to observe that 
\begin{equation}\label{thm-eq-1}
    (R-Sb)\cup (Sb-R)=\mathcal{F}_{2n}.
\end{equation}
\noindent\textbf{Claim 1:} $\mathrm{S}_{\mathrm{Dic}_{4n}}(2n,A)\geq (n+1)\bigl(7\cdot 2^{n-3}-4\bigr).$\par
To prove the claim, we require the sets $A$ for which $|A+A|>|A-A|$, that is, $$|(R+R)\cup (Sb+Sb)\cup (R+Sb)\cup (Sb+R)|>|(R-R)\cup (Sb-Sb)\cup (R-Sb)\cup (Sb-R)|.$$ By \eqref{thm-eq-1}, we have $$  |(R-Sb)\cup (Sb-R)|=2n.$$ Therefore, it is enough to show that
\begin{align}\label{(5.2)}
 |(R+R)\cup (Sb+Sb)|>|(R-R)\cup (Sb-Sb)|   
\end{align}
and 
\begin{align}\label{(5.3)}
|(R+Sb)\cup (Sb+R)|=2n.    
\end{align}
Since $$R+R=\mathcal{R}_{2r}\setminus\{a^{2r+2n-1}\}, \text{ and } R-R=\mathcal{R}_0\setminus\{a^{n}\},$$ 
to show \eqref{(5.2)}, it suffices to show that
$$ a^{2r+2n-1}\in Sb+Sb=\{a^{i_j-i_k+n}:1\leq j,k\leq n\}$$ and $$a^{n}\notin Sb-Sb=\{a^{i_j-i_k}:1\leq j,k\leq n\}.$$
This is equivalent to showing that for fixed $0\leq r\leq n$, the elements in $\mathcal{S}=\{i_1,i_2,\ldots,i_n\}\subset  \mathbb{Z}/2n\mathbb{Z}$ satisfy
\begin{align}
i_j-i_k\equiv 2r+n-1\pmod{2n}~\text{ for some } 1\leq j,k\leq n, \label{5.2}\\
i_j-i_k\not\equiv n \pmod{2n}\quad ~\text{ for all } 1\leq j,k\leq n. \label{5.3}
\end{align}
For fixed $0\leq r\leq n$, by Lemmas  \ref{Lemma 2.2}, \ref{Lemma 2.3}, \ref{Lemma 2.5}, and Remark \ref{remark 2.3}, for any integer $n\ge 6$ there are at least $7\cdot 2^{n-3}-4$ and at most $2^{n}$ subsets $\mathcal{S}$ satisfying \eqref{(5.3)}, \eqref{5.2} and \eqref{5.3}. Since $r$ ranges over $n+1$ values, it follows that the number of MSTD subsets is at least $(n+1)\bigl(7\cdot 2^{n-3}-4\bigr)$ and at most $2^{n}(n+1)$.

\noindent\textbf{Claim 2:}$$\mathrm{D}_{\mathrm{Dic}_{4n}}(2n, A) \geq 2(n+1) \quad \text{if } n \text{ is even},$$
$$\mathrm{D}_{\mathrm{Dic}_{4n}}(2n, A) = 0 \quad \text{if } n \text{ is odd}.$$
It is clear that if $a^{n}\in Sb-Sb=\{a^{i_j-i_k}:1\leq j,k\leq n\}$, from Lemma \ref{Lemma 2.4} we get
$$|R+Sb|=|(R+Sb)\cup(Sb+R)|=2n \text{ and } |Sb-R|=|(R-Sb)\cup (Sb-R)|=2n.$$ Since $R+R=\mathcal{R}_{2r}\setminus\{a^{2r+2n-1}\}, R-R=\mathcal{R}_0\setminus\{a^{n}\}$, to prove our claim, we need to show that
$$ a^{2r+2n-1}\notin Sb+Sb=\{a^{i_j-i_k+n}:1\leq j,k\leq n\}$$ and $$a^{n}\in Sb-Sb=\{a^{i_j-i_k}:1\leq j,k\leq n\}.$$

This is equivalent to showing that for fixed $0\leq r\leq n$,  the elements in $\mathcal{S}=\{i_1,i_2,\ldots,i_n\}\subset  \mathbb{Z}/2n\mathbb{Z}$ satisfy
\begin{align}
i_j-i_k\not\equiv 2r+n-1\pmod{2n}~\text{ for all } 1\leq j,k\leq n, \label{c31}\\
i_j-i_k\equiv n \pmod{2n}\quad ~\text{ for some } 1\leq j,k\leq n, \label{c32}
\end{align}
 Now consider the following cases. \begin{enumerate}
    \item  If $n$ is even, we take the sets $\{0,2,4,\ldots,2n\}$ and $\{1,3,\ldots,2n-1\}$. Observe that (\ref{c32}) holds. Moreover, for any $i_j,i_k$ either of these sets, we have $i_j-i_k-2r=2\big(\frac{i_j-i_k}{2}-r\big),$ and hence $i_j-i_k-2r\not\equiv n-1 \pmod{2n}$. 
    Therefore, these choices of $\mathcal{S}$ satisfy both \eqref{c31} and \eqref{c32}. Therefore, as $r$ ranges over $n+1$ values, it follows that $ \mathrm{D}_{\mathrm{Dic}_{4n}}(2n,A)\geq 2(n+1)$. 
     \item If $n$ is odd, then  Lemma \ref{Lemma 2.3}  implies that no subset $\mathcal{S}$  satisfies \eqref{c31}. Hence, MDTS sets do not exist for odd $n$, that is, $ \mathrm{D}_{\mathrm{Dic}_{4n}}(2n,A)=0$.\end{enumerate}
    
\noindent\textbf{Claim 3:} $$\mathrm{B}_{\mathrm{Dic}_{4n}}(2n,A)\geq(n+1)\sum_{\substack{i=3\\i\text{~odd}}}^{n+3}\binom{2n-i}{n-3} \quad \quad \text{if } n \text{ is even},$$ $$\mathrm{B}_{\mathrm{Dic}_{4n}}(2n,A)\geq \binom{2n}{n}(n+1)-2^{n}(n+1) \quad \text{if } n \text{ is odd}.$$
For fixed $0\leq r\leq n$,  consider a subset $\mathcal{S}=\{i_1,i_2,\ldots,i_n\}\subset  \mathbb{Z}/2n\mathbb{Z}$ for which
\begin{align}
i_j-i_k\equiv 2r+n-1\pmod{2n}~\text{ for some } 1\leq j,k\leq n, \label{c11}\\
i_j-i_k\equiv n \pmod{2n}\quad ~\text{ for some } 1\leq j,k\leq n. \label{c22}
\end{align}
This implies $$(R+R)\cup(Sb+Sb)=\mathcal{R}_{2r} \text{ and } (R-R)\cup (Sb-Sb)= \mathcal{R}_0$$ and therefore  $$|(R+R)\cup(Sb+Sb)|=2n \text{ and } |(R-R)\cup (Sb-Sb)|=2n.$$  
From Lemma \ref{Lemma 2.4} we have $|R+Sb|=|(R+Sb)\cup(Sb+R)|=2n \text{ and } |Sb-R|=|(R-Sb)\cup (Sb-R)|=2n$, which implies that the corresponding set is balanced. 

The following families of 
$n$-subsets satisfy (\ref{c11}), (\ref{c22}) and therefore contribute to the counting lower bound for even $n$.
 \begin{enumerate}
    \item The subsets of the form  $\{0,n+1-2r \pmod{2n}, n,i_4,i_5,\ldots, i_n\}$, where $\{i_4,\ldots, i_n\}$  are chosen from  the  remaining $2n-3$ elements, contribute  $\binom{2n-3}{n-3}$ sets.

\item The subsets of the form  $\{1,2r+n \pmod{2n},n+1 ,i_4, i_5,\ldots, i_n\}$, where $\{0,n\}\not\in\{i_4,\ldots, i_n\}$ are chosen from  the  remaining $2n-5$ elements, contribute  $\binom{2n-5}{n-3}$ sets.

\item The subsets of the form $\mathcal{S}=\{2,n+3-2r \pmod{2n},n+2 ,i_4, i_5,\ldots, i_n\}$, where $\{0,n,1,n+1\}\not\in \{i_4,\ldots, i_n\}$ are chosen from  the  remaining $2n-7$ elements, contribute $\binom{2n-7}{n-3}$ sets .
$$\vdots\quad \quad \quad \vdots\quad \quad \quad \vdots$$
\item[(n/2)] The subsets of the form $\mathcal{S}=\{n/2,2r+3n/2-1 \pmod{2n},3n/2 ,i_4, i_5,\ldots, i_n\}$, where $\{0,n,1,n+1,2,n+2,\ldots,n/2-1,3n/2-1\}\not\in\{i_4,\ldots, i_n\}$ are chosen from  the  remaining $n-3$ elements, contribute $\binom{n-3}{n-3}$ sets.
\end{enumerate}
Therefore, as $r$ ranges over $n+1$ values, we obtain
$$\mathrm{B_{\mathrm{Dic_{4n}}}(2n,A)}\geq (n+1) \sum_{\substack{i=3\\i\text{~odd}}}^{n+3}\binom{2n-i}{n-3}.$$
For odd $n$, as $r$ ranges over $n+1$ values and since there are at most $2^n(n+1)$ MSTD sets, $0$ MDTS sets, we obtain 
$$\mathrm{B}_{\mathrm{Dic}_{4n}}(2n,A)\geq \binom{2n}{n}(n+1)-2^{n}(n+1).$$

This completes the proof.
\end{proof}
\begin{proof}[Proof of Theorem \ref{Theorem 1.9}]
The proof follows immediately from Theorem \ref{Theorem1.7}.
\end{proof}

\section{Concluding remarks and future research}\label{section-6}
In this paper, we count the exact numbers of MSTD, MDTS, and balanced subsets of size two. For odd values of $n$, we also obtain exact counts of these sets for subsets of size three, where the results depend on whether 
$n$ is divisible by $3$. Moreover, using the same techniques, one can also derive exact formulas for subsets of size three when $n$ is even. A natural direction for future work is to investigate subsets of size four and larger. However, the method used in this paper involves many detailed case analyses and quickly becomes difficult to handle for larger sets. Finding new and simpler approaches may help to understand these cases more effectively.\par
\noindent The following problems and conjectures naturally arise from our study. The results of Theorems  \ref{Theorem1.0} and  \ref{Theorem1.1} imply that, for $m=2,3$, $\gcd\big\{ \mathrm{S}_{\mathrm{Dic}_{4n}}(m): n\geq3 \text{~~is odd}  \big\}$ is divisible by $4$. This observation motivates the following question.
\begin{question}
For any fixed integer $m$ with $2 \leq m \leq 2n$, is $4$ a divisor of $\gcd\big\{ \mathrm{S}_{\mathrm{Dic}_{4n}}(m): n\geq3  \big\}$? 
\end{question}
Furthermore, Theorems \ref{Theorem1.0} and  \ref{Theorem1.1} show  that for any fixed odd integer $n\geq3$, $4n$ is a divisor of $\gcd\big\{ \mathrm{S}_{\mathrm{Dic}_{4n}}(m): m=2,3  \big\}$. This observation naturally leads to the following question.
\begin{question}
For any fixed integer $n\geq3$ , is $4n$ a divisor of $\gcd\big\{ \mathrm{S}_{\mathrm{Dic}_{4n}}(m): m\geq2  \big\}$?
\end{question}
In addition, Theorems \ref{Theorem1.0} and \ref{Theorem1.1}, imply that $\mathrm{S}_{\mathrm{Dic}_{4n}}(m),\mathrm{D}_{\mathrm{Dic}_{4n}}(m) \text{ and }\mathrm{B}_{\mathrm{Dic}_{4n}}(m)$ are even integers for any odd integer $n\geq 3$ and $m=2,3$. This observation suggests the following more general question.
\begin{question}
Is it true that $\mathrm{S}_{\mathrm{Dic}_{4n}}(m),\mathrm{D}_{\mathrm{Dic}_{4n}}(m)  \text{ and }\mathrm{B}_{\mathrm{Dic}_{4n}}(m)$ are even integers for any positive integers $n,m$ with $n\geq 3$ and $2\leq m\leq 2n$?
\end{question}
In light of the results established in this paper, we are led to propose the following two conjectures.\begin{conjecture}
For a given integer $m\geq2$, there exists a positive integer $K_m$ such that if $n\geq K_m$, then $\mathrm{Dic}_{4n}$ has more MSTD sets than MDTS sets of size $m$.
\end{conjecture}
\begin{conjecture}
For given  integer $m\geq2$, there exists a positive integer $K_m$ such that if $n\geq K_m$, then $ \mathrm{S}_{\mathrm{Dic}_{4n}}(m)> 2^{m/2}$.
\end{conjecture}

\section{Data Availability} 	
The authors confirm that their manuscript has no associated data.

\section{Competing Interests}
The authors confirm that they have no competing interest. 

\section*{Acknowledgment}
This work was initiated during the AIS in Advanced Combinatorics (2025) held at Ahmedabad University, India. The authors thank the organizers. The first author thanks the University Grants Commission (UGC), Government of India, for financial support through the award of the Junior Research Fellowship (JRF), Ref. No. 241620111598.
\bibliographystyle{amsplain}

\end{document}